\documentclass{amsart}
\usepackage{amssymb,amsmath,amsthm}
\usepackage[dvips]{graphicx}

\usepackage{graphicx}
\usepackage[dvips]{epsfig}

\let\oldmarginpar\marginpar
\renewcommand\marginpar[1]{\-\oldmarginpar[\raggedleft\footnotesize #1]%
{\raggedright\footnotesize #1}}

\begin{document}

\newtheorem{theorem}{Theorem}[section]
\newtheorem{corollary}[theorem]{Corollary}
\newtheorem{lemma}[theorem]{Lemma}
\newtheorem{proposition}[theorem]{Proposition}
\theoremstyle{definition}
\newtheorem{definition}[theorem]{Definition}
\theoremstyle{remark}
\newtheorem{remark}[theorem]{Remark}
\theoremstyle{definition}
\newtheorem{example}[theorem]{Example}

\def\rank{{\text{rank}\,}}

\numberwithin{equation}{section}

\title[On the estimates of warping functions]{On the estimates of warping functions on isometric immersions}

\author{Kwang-Soon Park}
\address{Division of General Mathematics, Room 4-107, Changgong Hall, University of Seoul, Seoul 02504, Republic of Korea}
\email{parkksn@gmail.com}



\begin{abstract}
Using the results of \cite{P1}, we get some estimates of warping
functions for isometric immersions by changing the target manifolds
by some types of Riemannian manifolds: constant space forms and
Hermitian symmetric spaces. And we deal with equality cases and
obtain their applications. Finally, we give some open problems.
\end{abstract}

\maketitle
\section{Introduction}\label{intro}
\addcontentsline{toc}{section}{Introduction}

Let $(B, g_B)$ and $(F, g_F)$ be Riemannian manifolds. Given a
warped product manifold $M = B\times_f F$ with a warping function
$f$ (See \cite{P1}), we can consider an isometric immersion $\psi :
M \mapsto (\overline{M}, \overline{g})$, where $(\overline{M},
\overline{g})$ is a Riemannian manifold.

In 2018, B. Y. Chen \cite{C} proposed two Fundamental Questions on
the isometric immersion $\psi : M \mapsto (\overline{M},
\overline{g})$ and gave some recent results on these problems where
$(\overline{M}, \overline{g})$ is a K\"{a}hler manifold.

In 2014, as a generalization of Chen's works (\cite{C2},\cite{C3}),
the author \cite{P1} obtained two inequalities, which give the upper
bound and the lower bound of the function $\frac{\triangle f}{f}$.
Replacing the Riemannian manifold $(\overline{M}, \overline{g})$
with several types of Riemannian manifolds (i.e., real space forms,
complex space forms, quaternionic space forms, Sasakian space forms,
Kenmotsu space forms, Hermitian symmetric spaces: complex two-plane
Grassmannians, complex hyperbolic two-plane Grassmannians, complex
quadrics), we will obtain the upper bounds and the lower bounds of
the functions $\frac{\triangle f}{f}$. And by using these results,
we will get some equality cases of these relations and obtain their
applications.

We also know that warped product manifolds take an important
position in differential geometry and in physics, in particular in
general relativity. And Nash's result \cite{N} implies that each
warped product manifold can be isometrically embedded in a Euclidean
space.

The paper is organized as follows. In section 2 we remind some
notions, which will be used in the following sections. In section 3
we estimate the upper bounds and the lower bounds of the functions
$\frac{\triangle f}{f}$ for constant space forms $(\overline{M},
\overline{g})$ and have some equality cases and their applications.
In section 4 we do the works for Hermitian symmetric spaces
$(\overline{M}, \overline{g})$. In section 5 we give some open
problems.

\section{Preliminaries}\label{prelim}

In this section we recall some notions, which will be used in the
following sections.

Let $(\overline{M}, \overline{g})$ be an $n$-dimensional Riemannian
manifold and let $M$ be an $m$-dimensional submanifold of
$(\overline{M}, \overline{g})$. We denote by $\overline{\nabla}$ and
$\nabla$ the Levi-Civita connections of $\overline{M}$ and $M$,
respectively.

Then we get the {\em Gauss formula} and the {\em Weingarten formula}
\begin{eqnarray}
\overline{\nabla}_X Y &=& \nabla_X Y + h(X,Y),   \label{rel:2-l} \\
\overline{\nabla}_X N &=& -A_N X + D_X N,   \label{rel:2-2}
\end{eqnarray}
respectively, for tangent vector fields $X,Y\in \Gamma(TM)$ and a
normal vector field $N\in \Gamma(TM^{\perp})$, where $h$, $A$, $D$
denote the {\em second fundamental form}, the {\em shape operator},
the {\em normal connection} of $M$ in $\overline{M}$, respectively.

Then we know
\begin{equation}\label{rel:2-3}
\overline{g}(A_N X, Y) = \overline{g}(h(X,Y), N).
\end{equation}
Fix a local orthonormal frame $\{ v_1, \cdots, v_n \}$ of
$T\overline{M}$ with $v_i\in \Gamma(TM)$, $1\leq i \leq m$ and
$v_{\alpha}\in \Gamma(TM^{\perp})$, $m+1\leq \alpha \leq n$. We
define the mean curvature vector field $H$, the squared mean
curvature $H^2$, the squared norm $|| h ||^2$ of the second
fundamental form $h$ as follows:
\begin{eqnarray}
H &=& \frac{1}{m}\text{trace} h = \frac{1}{m}\sum_{i=1}^m h(v_i,v_i),   \label{rel:2-4} \\
H^2 &=& \overline{g}(H, H),   \label{rel:2-5}   \\
|| h ||^2 &=& \sum_{i,j=1}^m \overline{g}(h(v_i,v_j), h(v_i,v_j)).
\label{rel:2-6}
\end{eqnarray}
We call the submanifold $M \subset (\overline{M}, \overline{g})$
{\em totally geodesic} if the second fundamental form $h$ vanishes
identically. Denote by $R$, $\overline{R}$ the Riemannian curvature
tensors of $M$, $\overline{M}$, respectively.

Let
\begin{eqnarray}
K(X\wedge Y) &:=& \frac{g(R(X,Y)Y, X)}{g(X, X)g(Y, Y) - g(X, Y)^2},   \nonumber \\
\overline{K}(X\wedge Y) &:=& \frac{\overline{g}(\overline{R}(X,Y)Y,
X)}{\overline{g}(X, X)\overline{g}(Y, Y) - \overline{g}(X, Y)^2}
\nonumber
\end{eqnarray}
for $X,Y\in \Gamma(TM)$, where $g$ denotes the induced metric on $M$
of $(\overline{M}, \overline{g})$. i.e., given a plane $V \subset
T_p M$, $p\in M$, spanned by vectors $X,Y\in T_p M$, $K(V) =
K(X\wedge Y)$ and $\overline{K}(V) = \overline{K}(X\wedge Y)$ denote
the sectional curvatures of a plane $V$ in $M$ and in
$\overline{M}$, respectively.

Let
\begin{eqnarray}
(\inf \overline{K})(p) &:=& \inf \{ \overline{K}(V) \mid V \subset T_p M, \dim V = 2 \},   \label{rel:2-7} \\
(\sup \overline{K})(p) &:=& \sup \{ \overline{K}(V) \mid V \subset
T_p M, \dim V = 2 \}. \label{rel:2-8}
\end{eqnarray}
Let $\overline{R}(X,Y,Z,W) := \overline{g}(\overline{R}(X,Y)Z, W)$
for $X,Y,Z,W\in \Gamma(T\overline{M})$.

Given a $C^{\infty}-$function $f\in C^{\infty}(M)$, we define the
{\em Laplacian} $\triangle f$ of $f$ by
\begin{equation}\label{rel:2-9}
\triangle f := \sum_{i=1}^m ((\nabla_{v_i} v_i)f - v_i^2 f).
\nonumber
\end{equation}
Let $(B, g_B)$ and $(F, g_F)$ be Riemannian manifolds.

Throughout this paper, we will denote by $(M, g) := (B\times_f F,
g_B + f^2 g_F)$ the warped product manifold of Riemannian manifolds
$(B, g_B)$ and $(F, g_F)$ with the warping function $f : B \mapsto
\mathbb{R}^+$ (See \cite{P1}).

By Theorem 3.1, Theorem 3.4, and their proofs of \cite{P1}, we have

\begin{lemma}\label{lem:l}
Let $(M, g) = (B\times_f F, g_B + f^2 g_F)$ be a warped product
manifold and let $(\overline{M}, \overline{g})$ be a Riemannian
manifold. Let $\psi : (M, g)\mapsto (\overline{M}, \overline{g})$ be
an isometric immersion. Then we get
\begin{equation}\label{rel:3-l}
\frac{m_1 m^2}{2(m-1)} H^2 - \frac{m_1}{2} || h ||^2 + m_1 \inf
\overline{K} \leq \frac{\triangle f}{f} \leq \frac{m^2}{4m_2} H^2 +
m_1 \sup \overline{K},
\end{equation}
where $m_1 = \dim B$ and $m_2 = \dim F$ with $m=m_1+m_2$.
\end{lemma}

\section{Constant space forms}\label{esti}

In this section, we will estimate the functions $\frac{\triangle
f}{f}$ for isometric immersions $\psi : (M, g) = (B\times_f F, g_B +
f^2 g_F) \mapsto (\overline{M}, \overline{g})$ with constant space
forms $(\overline{M}, \overline{g})$. We also deal with  equality
cases and obtain their applications.

Using Lemma \ref{lem:l}, we obtain

\begin{theorem}\label{thm:l}
Let $(M, g) = (B\times_f F, g_B + f^2 g_F)$ be a warped product
manifold and $(\overline{M}, \overline{g}) = (\overline{M}(c),
\overline{g})$ a real space form of constant sectional curvature
$c$. Let $\psi : (M, g)\mapsto (\overline{M}, \overline{g})$ be an
isometric immersion. Then we have
\begin{equation}\label{rel:3-2}
\frac{m_1 m^2}{2(m-1)} H^2 - \frac{m_1}{2} || h ||^2 + m_1 c \leq
\frac{\triangle f}{f} \leq \frac{m^2}{4m_2} H^2 + m_1 c,
\end{equation}
where $m_1 = \dim B$ and $m_2 = \dim F$ with $m=m_1+m_2$.
\end{theorem}

\begin{proof}
We know that the Riemannian curvature tensor $\overline{R}$
\cite{KN} of $(\overline{M}, \overline{g})$ is given by
\begin{equation}\label{rel:3-3}
\overline{R}(X,Y)Z = c(\overline{g}(Y,Z)X - \overline{g}(X,Z)Y)
\end{equation}
for $X,Y,Z\in \Gamma(T\overline{M})$. Since $\inf \overline{K} =
\sup \overline{K} = c$, by Lemma \ref{lem:l}, we get the result.
\end{proof}

Then we easily obtain

\begin{corollary}
Let $(M, g) = (B\times_f F, g_B + f^2 g_F)$ be a warped product
manifold and $(\overline{M}, \overline{g}) = (\overline{M}(c),
\overline{g})$ a real space form of constant sectional curvature
$c$. Let $\psi : (M, g)\mapsto (\overline{M}, \overline{g})$ be an
isometric immersion. Assume that the manifold $(M, g)$ is a totally
geodesic submanifold of $(\overline{M}, \overline{g})$. Then we get
$$
m_1 c \leq \frac{\triangle f}{f} \leq m_1 c \ .
$$
\end{corollary}

\begin{remark}
Let $(M, g) = (B\times_f F, g_B + f^2 g_F)$ be a warped product
manifold and $(\overline{M}, \overline{g}) = (\overline{M}(c),
\overline{g})$ a real space form of constant sectional curvature
$c$. Let $\psi : (M, g)\mapsto (\overline{M}, \overline{g})$ be an
isometric immersion. Assume that the manifold $(M, g)$ is a totally
geodesic submanifold of $(\overline{M}, \overline{g})$.

Then the warping function $f$ is an eigen-function with eigenvalue
$m_1 c$.

In particular, if $c = 0$ (i.e., $(\overline{M}, \overline{g})$ is a
Euclidean space $\mathbb{E}^n$.), then the warping function $f$ is a
harmonic function.
\end{remark}

\begin{lemma}
Let $(M, g) = (B\times_f F, g_B + f^2 g_F)$ be a warped product
manifold and $(\overline{M}, \overline{g}) = (\overline{M}(c),
\overline{g})$ a real space form of constant sectional curvature
$c$. Let $\psi : (M, g)\mapsto (\overline{M}, \overline{g})$ be an
isometric immersion.

There does not exist a totally geodesic submanifold $(M, g)$ of
$(\overline{M}, \overline{g})$ such that either the warping function
$f$ is not an eigen-function or the eigenvalue of $f$ is not equal
to $m_1 c$.
\end{lemma}

\begin{theorem}\label{thm:2}
Let $(M, g) = (B\times_f F, g_B + f^2 g_F)$ be a warped product
manifold and $(\overline{M}, \overline{g}) = (\overline{M}(c),
\overline{g}, J)$ a complex space form of constant holomorphic
sectional curvature $c$. Let $\psi : (M, g)\mapsto (\overline{M},
\overline{g})$ be an isometric immersion. Then we have
\begin{eqnarray}
\frac{m_1 m^2}{2(m-1)} H^2 - \frac{m_1}{2} || h ||^2 + m_1
\frac{c}{4} \leq
\frac{\triangle f}{f} \leq \frac{m^2}{4m_2} H^2 + m_1 c, & c\geq 0,   \label{rel:3-4} \\
\frac{m_1 m^2}{2(m-1)} H^2 - \frac{m_1}{2} || h ||^2 + m_1 c \leq
\frac{\triangle f}{f} \leq \frac{m^2}{4m_2} H^2 + m_1 \frac{c}{4}, &
c< 0, \label{rel:3-5}
\end{eqnarray}
where $m_1 = \dim B$ and $m_2 = \dim F$ with $m=m_1+m_2$.
\end{theorem}

\begin{proof}
We see that the Riemannian curvature tensor $\overline{R}$ \cite{KN}
of $(\overline{M}, \overline{g})$ is given by
\begin{eqnarray}
&&\overline{R}(X,Y)Z \label{rel:3-6}  \\
&&= \frac{c}{4}(\overline{g}(Y,Z)X - \overline{g}(X,Z)Y +
\overline{g}(JY, Z)JX - \overline{g}(JX, Z)JY - 2\overline{g}(JX,
Y)JZ)  \nonumber
\end{eqnarray}
for $X,Y,Z\in \Gamma(T\overline{M})$. Given orthonormal vectors
$X,Y\in T_p \overline{M}$, $p\in \overline{M}$, we get
\begin{equation}\label{rel:3-7}
\overline{K}(X\wedge Y) = \overline{R}(X,Y,Y,X) = \frac{c}{4} (1 +
3\overline{g}(JX, Y)^2)
\end{equation}
so that since $0\leq |\overline{g}(JX, Y)|\leq 1$, we easily obtain
\begin{eqnarray}
\frac{c}{4} \leq \overline{K}(X\wedge Y) \leq  c, & c\geq 0,   \nonumber \\
c \leq \overline{K}(X\wedge Y) \leq \frac{c}{4}, & c< 0. \nonumber
\end{eqnarray}
From Lemma \ref{lem:l}, the result follows.
\end{proof}

\begin{corollary}
Let $(M, g) = (B\times_f F, g_B + f^2 g_F)$ be a warped product
manifold and $(\overline{M}, \overline{g}) = (\overline{M}(c),
\overline{g}, J)$ a complex space form of constant holomorphic
sectional curvature $c$. Let $\psi : (M, g)\mapsto (\overline{M},
\overline{g})$ be an isometric immersion. Assume that the manifold
$(M, g)$ is a totally geodesic totally real submanifold of
$(\overline{M}, \overline{g})$ (i.e., $J(TM) \subset TM^{\perp}$).

Then we have
$$
m_1 \frac{c}{4} \leq \frac{\triangle f}{f} \leq m_1 \frac{c}{4}.
$$
\end{corollary}

\begin{proof}
By Lemma \ref{lem:l} and (\ref{rel:3-7}), we obtain the result.
\end{proof}

\begin{remark}
Let $(M, g) = (B\times_f F, g_B + f^2 g_F)$ be a warped product
manifold and $(\overline{M}, \overline{g}) = (\overline{M}(c),
\overline{g}, J)$ a complex space form of constant holomorphic
sectional curvature $c$. Let $\psi : (M, g)\mapsto (\overline{M},
\overline{g})$ be an isometric immersion. Assume that the manifold
$(M, g)$ is a totally geodesic totally real submanifold of
$(\overline{M}, \overline{g})$.

Then the warping function $f$ is an eigen-function with eigenvalue
$\frac{m_1 c}{4}$.
\end{remark}

\begin{lemma}
Let $(M, g) = (B\times_f F, g_B + f^2 g_F)$ be a warped product
manifold and $(\overline{M}, \overline{g}) = (\overline{M}(c),
\overline{g}, J)$ a complex space form of constant holomorphic
sectional curvature $c$. Let $\psi : (M, g)\mapsto (\overline{M},
\overline{g})$ be an isometric immersion.

There does not exist a totally geodesic totally real submanifold
$(M, g)$ of $(\overline{M}, \overline{g})$ such that either the
warping function $f$ is not an eigen-function or the eigenvalue of
$f$ is not equal to $\frac{m_1 c}{4}$.
\end{lemma}

\begin{corollary}
Let $(M, g) = (B\times_f F, g_B + f^2 g_F)$ be a warped product
manifold and $(\overline{M}, \overline{g}) = (\overline{M}(c),
\overline{g}, J)$ a complex space form of constant holomorphic
sectional curvature $c$. Let $\psi : (M, g)\mapsto (\overline{M},
\overline{g})$ be an isometric immersion. Assume that the manifold
$(M, g)$ is a 2-dimensional totally geodesic complex submanifold of
$(\overline{M}, \overline{g})$ (i.e., $J(TM) = TM$).

Then we have
$$
m_1 c \leq \frac{\triangle f}{f} \leq m_1 c.
$$
\end{corollary}

\begin{remark}
Let $(M, g) = (B\times_f F, g_B + f^2 g_F)$ be a warped product
manifold and $(\overline{M}, \overline{g}) = (\overline{M}(c),
\overline{g}, J)$ a complex space form of constant holomorphic
sectional curvature $c$. Let $\psi : (M, g)\mapsto (\overline{M},
\overline{g})$ be an isometric immersion. Assume that the manifold
$(M, g)$ is a 2-dimensional totally geodesic complex submanifold of
$(\overline{M}, \overline{g})$.

Then the warping function $f$ is an eigen-function with eigenvalue
$m_1 c$.
\end{remark}

\begin{lemma}
Let $(M, g) = (B\times_f F, g_B + f^2 g_F)$ be a warped product
manifold and $(\overline{M}, \overline{g}) = (\overline{M}(c),
\overline{g}, J)$ a complex space form of constant holomorphic
sectional curvature $c$. Let $\psi : (M, g)\mapsto (\overline{M},
\overline{g})$ be an isometric immersion.

There does not exist a 2-dimensional totally geodesic complex
submanifold $(M, g)$ of $(\overline{M}, \overline{g})$ such that
either the warping function $f$ is not an eigen-function or the
eigenvalue of $f$ is not equal to $m_1 c$.
\end{lemma}

\begin{theorem}\label{thm:3}
Let $(M, g) = (B\times_f F, g_B + f^2 g_F)$ be a warped product
manifold and $(\overline{M}, \overline{g}) = (\overline{M}(c), E,
\overline{g})$ a quaternionic space form of constant quaternionic
sectional curvature $c$. Let $\psi : (M, g)\mapsto (\overline{M},
\overline{g})$ be an isometric immersion. Then we obtain
\begin{eqnarray}
\frac{m_1 m^2}{2(m-1)} H^2 - \frac{m_1}{2} || h ||^2 + m_1
\frac{c}{4} \leq
\frac{\triangle f}{f} \leq \frac{m^2}{4m_2} H^2 + m_1 c, & c\geq 0,   \label{rel:3-8} \\
\frac{m_1 m^2}{2(m-1)} H^2 - \frac{m_1}{2} || h ||^2 + m_1 c \leq
\frac{\triangle f}{f} \leq \frac{m^2}{4m_2} H^2 + m_1 \frac{c}{4}, &
c< 0, \label{rel:3-9}
\end{eqnarray}
where $m_1 = \dim B$ and $m_2 = \dim F$ with $m=m_1+m_2$.
\end{theorem}

\begin{proof}
We know that the Riemannian curvature tensor $\overline{R}$ \cite{I}
of $(\overline{M}, \overline{g})$ is given by
\begin{eqnarray}
&&\overline{R}(X,Y)Z = \frac{c}{4}(\overline{g}(Y,Z)X - \overline{g}(X,Z)Y  \label{rel:3-10}   \\
&&+ \sum_{\alpha=1}^3 (\overline{g}(J_{\alpha} Y, Z)J_{\alpha} X -
\overline{g}(J_{\alpha} X, Z)J_{\alpha} Y - 2\overline{g}(J_{\alpha}
X, Y)J_{\alpha} Z)) \nonumber
\end{eqnarray}
for $X,Y,Z\in \Gamma(T\overline{M})$. Given orthonormal vectors
$X,Y\in T_p \overline{M}$, $p\in \overline{M}$, we have
\begin{equation}\label{rel:3-11}
\overline{K}(X\wedge Y) = \overline{R}(X,Y,Y,X) = \frac{c}{4} (1 +
3\sum_{\alpha=1}^3 \overline{g}(J_{\alpha} X, Y)^2).
\end{equation}
Since $\{ J_1X,J_2X,J_3X \}$ is orthonormal, we get $0\leq
\sum_{\alpha=1}^3 \overline{g}(J_{\alpha} X, Y)^2 \leq |Y|^2 = 1$ so
that
\begin{eqnarray}
\frac{c}{4} \leq \overline{K}(X\wedge Y) \leq  c, & c\geq 0,   \nonumber \\
c \leq \overline{K}(X\wedge Y) \leq \frac{c}{4}, & c< 0. \nonumber
\end{eqnarray}
From Lemma \ref{lem:l}, we obtain the result.
\end{proof}

\begin{corollary}
Let $(M, g) = (B\times_f F, g_B + f^2 g_F)$ be a warped product
manifold and $(\overline{M}, \overline{g}) = (\overline{M}(c), E,
\overline{g})$ a quaternionic space form of constant quaternionic
sectional curvature $c$. Let $\psi : (M, g)\mapsto (\overline{M},
\overline{g})$ be an isometric immersion. Assume that the manifold
$(M, g)$ is a totally geodesic totally real submanifold of
$(\overline{M}, \overline{g})$ (i.e., $J_{\alpha} (TM) \subset
TM^{\perp}$, $\forall \alpha\in \{ 1,2,3 \}$).

Then we have
$$
m_1 \frac{c}{4} \leq \frac{\triangle f}{f} \leq m_1 \frac{c}{4}.
$$
\end{corollary}

\begin{proof}
By Lemma \ref{lem:l} and (\ref{rel:3-11}), we obtain the result.
\end{proof}

\begin{lemma}
Let $(M, g) = (B\times_f F, g_B + f^2 g_F)$ be a warped product
manifold and $(\overline{M}, \overline{g}) = (\overline{M}(c), E,
\overline{g})$ a quaternionic space form of constant quaternionic
sectional curvature $c$. Let $\psi : (M, g)\mapsto (\overline{M},
\overline{g})$ be an isometric immersion.

There does not exist a totally geodesic totally real submanifold
$(M, g)$ of $(\overline{M}, \overline{g})$ such that either the
warping function $f$ is not an eigen-function or the eigenvalue of
$f$ is not equal to $\frac{m_1 c}{4}$.
\end{lemma}

\begin{corollary}
Let $(M, g) = (B\times_f F, g_B + f^2 g_F)$ be a warped product
manifold and $(\overline{M}, \overline{g}) = (\overline{M}(c), E,
\overline{g})$ a quaternionic space form of constant quaternionic
sectional curvature $c$. Let $\psi : (M, g)\mapsto (\overline{M},
\overline{g})$ be an isometric immersion. Assume that the manifold
$(M, g)$ is a 4-dimensional totally geodesic quaternionic
submanifold of $(\overline{M}, \overline{g})$ (i.e., $J_{\alpha}(TM)
= TM$, $\forall \alpha\in \{ 1,2,3 \}$).

Then we have
$$
m_1 c \leq \frac{\triangle f}{f} \leq m_1 c.
$$
\end{corollary}

\begin{lemma}
Let $(M, g) = (B\times_f F, g_B + f^2 g_F)$ be a warped product
manifold and $(\overline{M}, \overline{g}) = (\overline{M}(c), E,
\overline{g})$ a quaternionic space form of constant quaternionic
sectional curvature $c$. Let $\psi : (M, g)\mapsto (\overline{M},
\overline{g})$ be an isometric immersion.

There does not exist a 4-dimensional totally geodesic quaternionic
submanifold $(M, g)$ of $(\overline{M}, \overline{g})$ such that
either the warping function $f$ is not an eigen-function or the
eigenvalue of $f$ is not equal to $m_1 c$.
\end{lemma}

\begin{theorem}\label{thm:4}
Let $(M, g) = (B\times_f F, g_B + f^2 g_F)$ be a warped product
manifold and $(\overline{M}, \overline{g}) = (\overline{M}(c), \phi,
\xi, \eta, \overline{g})$ a Sasakian space form of constant
$\phi$-sectional curvature $c$. Let $\psi : (M, g)\mapsto
(\overline{M}, \overline{g})$ be an isometric immersion. Then we
obtain
\begin{eqnarray}
\frac{m_1 m^2}{2(m-1)} H^2 - \frac{m_1}{2} || h ||^2 + m_1
 \leq
\frac{\triangle f}{f} \leq \frac{m^2}{4m_2} H^2 + m_1 c, & c\geq 1,   \label{rel:3-12} \\
\frac{m_1 m^2}{2(m-1)} H^2 - \frac{m_1}{2} || h ||^2 + m_1 c \leq
\frac{\triangle f}{f} \leq \frac{m^2}{4m_2} H^2 + m_1, & c< 1,
\label{rel:3-13}
\end{eqnarray}
where $m_1 = \dim B$ and $m_2 = \dim F$ with $m=m_1+m_2$.
\end{theorem}

\begin{proof}
We see that the Riemannian curvature tensor $\overline{R}$ \cite{O}
of $(\overline{M}, \overline{g})$ is given by
\begin{eqnarray}
&&\overline{R}(X,Y)Z = \frac{c+3}{4}(\overline{g}(Y,Z)X - \overline{g}(X,Z)Y)  \label{rel:3-14}   \\
&&+ \frac{c-1}{4} (\eta(X)\eta(Z)Y - \eta(Y)\eta(Z)X +
\eta(Y)\overline{g}(X,Z)\xi - \eta(X)\overline{g}(Y,Z)\xi
\nonumber   \\
&&+ \overline{g}(\phi Y, Z)\phi X - \overline{g}(\phi X, Z)\phi Y -
2\overline{g}(\phi X, Y)\phi Z) \nonumber
\end{eqnarray}
for $X,Y,Z\in \Gamma(T\overline{M})$. Given orthonormal vectors
$X,Y\in T_p \overline{M}$, $p\in \overline{M}$, we have
\begin{equation}\label{rel:3-15}
\overline{K}(X\wedge Y) = \overline{R}(X,Y,Y,X) = \frac{c+3}{4} +
\frac{c-1}{4} (-\eta(Y)^2 - \eta(X)^2 + 3\overline{g}(\phi X, Y)^2).
\end{equation}
If $\xi\in \text{Span}(X,Y)$, then $-\eta(Y)^2 - \eta(X)^2 +
3\overline{g}(\phi X, Y)^2 = -1$. If $Y = \phi X$ and $\eta(X) = 0$,
then $-\eta(Y)^2 - \eta(X)^2 + 3\overline{g}(\phi X, Y)^2 = 3$.
Hence we get $-1\leq -\eta(Y)^2 - \eta(X)^2 + 3\overline{g}(\phi X,
Y)^2 \leq 3$ so that
\begin{eqnarray}
1 \leq \overline{K}(X\wedge Y) \leq  c, & c\geq 1,   \nonumber \\
c \leq \overline{K}(X\wedge Y) \leq 1, & c< 1. \nonumber
\end{eqnarray}
From Lemma \ref{lem:l}, the result follows.
\end{proof}

\begin{corollary}
Let $(M, g) = (B\times_f F, g_B + f^2 g_F)$ be a warped product
manifold and $(\overline{M}, \overline{g}) = (\overline{M}(c), \phi,
\xi, \eta, \overline{g})$ a Sasakian space form of constant
$\phi$-sectional curvature $c$. Let $\psi : (M, g)\mapsto
(\overline{M}, \overline{g})$ be an isometric immersion. Assume that
the manifold $(M, g)$ is a totally geodesic $\phi$-totally real
submanifold of $(\overline{M}, \overline{g})$ with $\xi\in
\Gamma(TM^{\perp})$ (i.e., $\phi (TM) \subset TM^{\perp}$).

Then we get
$$
m_1 \frac{c+3}{4} \leq \frac{\triangle f}{f} \leq m_1 \frac{c+3}{4}.
$$
\end{corollary}

\begin{proof}
By Lemma \ref{lem:l} and (\ref{rel:3-15}), we obtain the result.
\end{proof}

\begin{lemma}
Let $(M, g) = (B\times_f F, g_B + f^2 g_F)$ be a warped product
manifold and $(\overline{M}, \overline{g}) = (\overline{M}(c), \phi,
\xi, \eta, \overline{g})$ a Sasakian space form of constant
$\phi$-sectional curvature $c$. Let $\psi : (M, g)\mapsto
(\overline{M}, \overline{g})$ be an isometric immersion.

There does not exist a totally geodesic $\phi$-totally real
submanifold $(M, g)$ of $(\overline{M}, \overline{g})$ such that
either the warping function $f$ is not an eigen-function or the
eigenvalue of $f$ is not equal to $\frac{m_1(c+3)}{4}$.
\end{lemma}

\begin{corollary}
Let $(M, g) = (B\times_f F, g_B + f^2 g_F)$ be a warped product
manifold and $(\overline{M}, \overline{g}) = (\overline{M}(c), \phi,
\xi, \eta, \overline{g})$ a Sasakian space form of constant
$\phi$-sectional curvature $c$. Let $\psi : (M, g)\mapsto
(\overline{M}, \overline{g})$ be an isometric immersion. Assume that
the manifold $(M, g)$ is a 2-dimensional totally geodesic
 submanifold of $(\overline{M}, \overline{g})$ with $\xi\in
\Gamma(TM)$.

Then we get
$$
m_1 \cdot 1 \leq \frac{\triangle f}{f} \leq m_1 \cdot 1
$$
\end{corollary}

\begin{lemma}
Let $(M, g) = (B\times_f F, g_B + f^2 g_F)$ be a warped product
manifold and $(\overline{M}, \overline{g}) = (\overline{M}(c), \phi,
\xi, \eta, \overline{g})$ a Sasakian space form of constant
$\phi$-sectional curvature $c$. Let $\psi : (M, g)\mapsto
(\overline{M}, \overline{g})$ be an isometric immersion.

There does not exist a 2-dimensional totally geodesic
 submanifold $(M, g)$ of $(\overline{M}, \overline{g})$ with $\xi\in
\Gamma(TM)$ such that either the warping function $f$ is not an
eigen-function or the eigenvalue of $f$ is not equal to $m_1$.
\end{lemma}

\begin{corollary}
Let $(M, g) = (B\times_f F, g_B + f^2 g_F)$ be a warped product
manifold and $(\overline{M}, \overline{g}) = (\overline{M}(c), \phi,
\xi, \eta, \overline{g})$ a Sasakian space form of constant
$\phi$-sectional curvature $c$. Let $\psi : (M, g)\mapsto
(\overline{M}, \overline{g})$ be an isometric immersion. Assume that
the manifold $(M, g)$ is a 2-dimensional totally geodesic
$\phi$-invariant submanifold of $(\overline{M}, \overline{g})$ with
$\xi\in \Gamma(TM^{\perp})$ (i.e., $\phi (TM) = TM$).

Then we have
$$
m_1 c \leq \frac{\triangle f}{f} \leq m_1 c
$$
\end{corollary}

\begin{lemma}
Let $(M, g) = (B\times_f F, g_B + f^2 g_F)$ be a warped product
manifold and $(\overline{M}, \overline{g}) = (\overline{M}(c), \phi,
\xi, \eta, \overline{g})$ a Sasakian space form of constant
$\phi$-sectional curvature $c$. Let $\psi : (M, g)\mapsto
(\overline{M}, \overline{g})$ be an isometric immersion.

There does not exist a 2-dimensional totally geodesic
$\phi$-invariant submanifold $(M, g)$ of $(\overline{M},
\overline{g})$ with $\xi\in \Gamma(TM^{\perp})$ such that either the
warping function $f$ is not an eigen-function or the eigenvalue of
$f$ is not equal to $m_1c$.
\end{lemma}

\begin{theorem}\label{thm:5}
Let $(M, g) = (B\times_f F, g_B + f^2 g_F)$ be a warped product
manifold and $(\overline{M}, \overline{g}) = (\overline{M}(c), \phi,
\xi, \eta, \overline{g})$ a Kenmotsu space form of constant
$\phi$-sectional curvature $c$. Let $\psi : (M, g)\mapsto
(\overline{M}, \overline{g})$ be an isometric immersion. Then we
obtain
\begin{eqnarray}
\frac{m_1 m^2}{2(m-1)} H^2 - \frac{m_1}{2} || h ||^2 - m_1
 \leq
\frac{\triangle f}{f} \leq \frac{m^2}{4m_2} H^2 + m_1 c, & c\geq -1,   \label{rel:3-16} \\
\frac{m_1 m^2}{2(m-1)} H^2 - \frac{m_1}{2} || h ||^2 + m_1 c \leq
\frac{\triangle f}{f} \leq \frac{m^2}{4m_2} H^2 - m_1, & c< -1,
\label{rel:3-17}
\end{eqnarray}
where $m_1 = \dim B$ and $m_2 = \dim F$ with $m=m_1+m_2$.
\end{theorem}

\begin{proof}
We know that the Riemannian curvature tensor $\overline{R}$ \cite{K}
of $(\overline{M}, \overline{g})$ is given by
\begin{eqnarray}
&&\overline{R}(X,Y)Z = \frac{c-3}{4}(\overline{g}(Y,Z)X - \overline{g}(X,Z)Y)  \label{rel:3-18}   \\
&&+ \frac{c+1}{4} (\eta(X)\eta(Z)Y - \eta(Y)\eta(Z)X +
\eta(Y)\overline{g}(X,Z)\xi - \eta(X)\overline{g}(Y,Z)\xi
\nonumber   \\
&&+ \overline{g}(\phi Y, Z)\phi X - \overline{g}(\phi X, Z)\phi Y -
2\overline{g}(\phi X, Y)\phi Z) \nonumber
\end{eqnarray}
for $X,Y,Z\in \Gamma(T\overline{M})$. Given orthonormal vectors
$X,Y\in T_p \overline{M}$, $p\in \overline{M}$, we have
\begin{equation}\label{rel:3-19}
\overline{K}(X\wedge Y) = \overline{R}(X,Y,Y,X) = \frac{c-3}{4} +
\frac{c+1}{4} (-\eta(Y)^2 - \eta(X)^2 + 3\overline{g}(\phi X, Y)^2).
\end{equation}
so that since  $-1\leq -\eta(Y)^2 - \eta(X)^2 + 3\overline{g}(\phi
X, Y)^2 \leq 3$, we get
\begin{eqnarray}
-1 \leq \overline{K}(X\wedge Y) \leq  c, & c\geq -1,   \nonumber \\
c \leq \overline{K}(X\wedge Y) \leq -1, & c< -1. \nonumber
\end{eqnarray}
From Lemma \ref{lem:l}, we obtain the result.
\end{proof}

\begin{corollary}
Let $(M, g) = (B\times_f F, g_B + f^2 g_F)$ be a warped product
manifold and $(\overline{M}, \overline{g}) = (\overline{M}(c), \phi,
\xi, \eta, \overline{g})$ a Kenmotsu space form of constant
$\phi$-sectional curvature $c$. Let $\psi : (M, g)\mapsto
(\overline{M}, \overline{g})$ be an isometric immersion. Assume that
the manifold $(M, g)$ is a 2-dimensional totally geodesic
 submanifold of $(\overline{M}, \overline{g})$ with $\xi\in
\Gamma(TM)$.

Then we get
$$
m_1 \cdot (-1) \leq \frac{\triangle f}{f} \leq m_1 \cdot (-1).
$$
\end{corollary}

\begin{lemma}
Let $(M, g) = (B\times_f F, g_B + f^2 g_F)$ be a warped product
manifold and $(\overline{M}, \overline{g}) = (\overline{M}(c), \phi,
\xi, \eta, \overline{g})$ a Kenmotsu space form of constant
$\phi$-sectional curvature $c$. Let $\psi : (M, g)\mapsto
(\overline{M}, \overline{g})$ be an isometric immersion.

There does not exist a 2-dimensional totally geodesic
 submanifold $(M, g)$ of $(\overline{M}, \overline{g})$ with $\xi\in
\Gamma(TM)$ such that either the warping function $f$ is not an
eigen-function or the eigenvalue of $f$ is not equal to $-m_1$.
\end{lemma}

\begin{corollary}
Let $(M, g) = (B\times_f F, g_B + f^2 g_F)$ be a warped product
manifold and $(\overline{M}, \overline{g}) = (\overline{M}(c), \phi,
\xi, \eta, \overline{g})$ a Kenmotsu space form of constant
$\phi$-sectional curvature $c$. Let $\psi : (M, g)\mapsto
(\overline{M}, \overline{g})$ be an isometric immersion. Assume that
the manifold $(M, g)$ is a 2-dimensional totally geodesic
$\phi$-invariant submanifold of $(\overline{M}, \overline{g})$ with
$\xi\in \Gamma(TM^{\perp})$ (i.e., $\phi (TM) = TM$).

Then we get
$$
m_1 c \leq \frac{\triangle f}{f} \leq m_1 c.
$$
\end{corollary}

\begin{lemma}
Let $(M, g) = (B\times_f F, g_B + f^2 g_F)$ be a warped product
manifold and $(\overline{M}, \overline{g}) = (\overline{M}(c), \phi,
\xi, \eta, \overline{g})$ a Kenmotsu space form of constant
$\phi$-sectional curvature $c$. Let $\psi : (M, g)\mapsto
(\overline{M}, \overline{g})$ be an isometric immersion.

There does not exist a 2-dimensional totally geodesic
$\phi$-invariant submanifold $(M, g)$ of $(\overline{M},
\overline{g})$ with $\xi\in \Gamma(TM^{\perp})$ such that either the
warping function $f$ is not an eigen-function or the eigenvalue of
$f$ is not equal to $m_1 c$.
\end{lemma}

\begin{corollary}
Let $(M, g) = (B\times_f F, g_B + f^2 g_F)$ be a warped product
manifold and $(\overline{M}, \overline{g}) = (\overline{M}(c), \phi,
\xi, \eta, \overline{g})$ a Kenmotsu space form of constant
$\phi$-sectional curvature $c$. Let $\psi : (M, g)\mapsto
(\overline{M}, \overline{g})$ be an isometric immersion. Assume that
the manifold $(M, g)$ is a totally geodesic $\phi$-totally real
submanifold of $(\overline{M}, \overline{g})$ with $\xi\in
\Gamma(TM^{\perp})$ (i.e., $\phi (TM) \subset TM^{\perp}$).

Then we have
$$
m_1 \frac{c-3}{4} \leq \frac{\triangle f}{f} \leq m_1 \frac{c-3}{4}.
$$
\end{corollary}

\begin{lemma}
Let $(M, g) = (B\times_f F, g_B + f^2 g_F)$ be a warped product
manifold and $(\overline{M}, \overline{g}) = (\overline{M}(c), \phi,
\xi, \eta, \overline{g})$ a Kenmotsu space form of constant
$\phi$-sectional curvature $c$. Let $\psi : (M, g)\mapsto
(\overline{M}, \overline{g})$ be an isometric immersion.

There does not exist a totally geodesic $\phi$-totally real
submanifold $(M, g)$ of $(\overline{M}, \overline{g})$ with $\xi\in
\Gamma(TM^{\perp})$ such that either the warping function $f$ is not
an eigen-function or the eigenvalue of $f$ is not equal to
$\frac{m_1(c-3)}{4}$.
\end{lemma}

\section{Hermitian symmetric spaces}\label{esti2}

\begin{theorem}\label{thm:6}
Let $(M, g) = (B\times_f F, g_B + f^2 g_F)$ be a warped product
manifold and $(\overline{M}, \overline{g}) = G_2 (\mathbb{C}^{m+2})
= SU_{m+2}/S(U_m U_2)$ the complex two-plane Grassmannian. Let $\psi
: (M, g)\mapsto (\overline{M}, \overline{g})$ be an isometric
immersion. Then we have
\begin{equation}\label{rel:3-20}
\frac{m_1 m^2}{2(m-1)} H^2 - \frac{m_1}{2} || h ||^2 - m_1  \leq
\frac{\triangle f}{f} \leq \frac{m^2}{4m_2} H^2 + 8m_1,
\end{equation}
where $m_1 = \dim B$ and $m_2 = \dim F$ with $m=m_1+m_2$.
\end{theorem}

\begin{proof}
We see that the Riemannian curvature tensor $\overline{R}$ \cite{P2}
of $(\overline{M}, \overline{g})$ is given by
\begin{eqnarray}
&&\overline{R}(X,Y)Z = \overline{g}(Y,Z)X - \overline{g}(X,Z)Y
\label{rel:3-21}   \\
&&+ \overline{g}(JY, Z)JX - \overline{g}(JX, Z)JY -
2\overline{g}(JX, Y)JZ \nonumber  \\
&&+ \sum_{\alpha=1}^3
(\overline{g}(J_{\alpha} Y, Z)J_{\alpha} X - \overline{g}(J_{\alpha}
X, Z)J_{\alpha} Y - 2\overline{g}(J_{\alpha}
X, Y)J_{\alpha} Z) \nonumber  \\
&&+ \sum_{\alpha=1}^3 (\overline{g}(J_{\alpha} JY, Z)J_{\alpha} JX -
\overline{g}(J_{\alpha} JX, Z)J_{\alpha} JY) \nonumber
\end{eqnarray}
for $X,Y,Z\in \Gamma(T\overline{M})$. Given orthonormal vectors
$X,Y\in T_p \overline{M}$, $p\in \overline{M}$, we get
\begin{eqnarray}
&&\overline{K}(X\wedge Y) = \overline{R}(X,Y,Y,X) = 1 +
3\overline{g}(JX, Y)^2 \label{rel:3-22}   \\
&&+ \sum_{\alpha=1}^3 (3\overline{g}(J_{\alpha} X, Y)^2 +
\overline{g}(J_{\alpha} JY, Y) \overline{g}(J_{\alpha} JX, X) -
\overline{g}(J_{\alpha} JX, Y)^2).  \nonumber
\end{eqnarray}
With simple computations, we obtain

$\overline{g}(JX, Y)^2 \leq |JX|^2 |Y|^2 = 1,$

$\displaystyle{ \sum_{\alpha=1}^3 \overline{g}(J_{\alpha} X, Y)^2
\leq |Y|^2 = 1}$ (since $\{ J_1 X,J_2 X,J_3 X \}$ is orthonormal),

\begin{eqnarray}
&&|\sum_{\alpha=1}^3 \overline{g}(J_{\alpha} JY, Y)
\overline{g}(J_{\alpha} JX, X)| \leq \sqrt{\sum_{\alpha=1}^3
\overline{g}(J_{\alpha} JY, Y)^2} \cdot \sqrt{\sum_{\alpha=1}^3
\overline{g}(J_{\alpha} JX, X)^2}   \nonumber \\
&&\leq \sqrt{|Y|^2} \sqrt{|X|^2} = 1 \nonumber
\end{eqnarray}
(by Cauchy-Schwarz inequality and since $\{ J_1JY,J_2JY,J_3JY \}$
and $\{ J_1JX,J_2JX,J_3JX \}$ are orthonormal)

$\Rightarrow \ \displaystyle{ -1 \leq \sum_{\alpha=1}^3
\overline{g}(J_{\alpha} JY, Y) \overline{g}(J_{\alpha} JX, X) \leq
1}$,

$\displaystyle{ \sum_{\alpha=1}^3 \overline{g}(J_{\alpha} JX, Y)^2
\leq |Y|^2 = 1}$ (since $\{ J_1JX,J_2JX,J_3JX \}$ is orthonormal).

By using the above relations, we obtain
\begin{equation}\label{rel:3-23}
\overline{K}(X\wedge Y) \leq 1 + 3 \cdot 1 + 3 \cdot 1 + 1 = 8.
\end{equation}
On the other hand, by the above relations, we have
\begin{eqnarray}
&&\overline{K}(X\wedge Y) \geq 1 + \sum_{\alpha=1}^3
(\overline{g}(J_{\alpha} JY, Y) \overline{g}(J_{\alpha} JX, X) -
\overline{g}(J_{\alpha} JX, Y)^2)  \label{rel:3-24}  \\
&&\geq 1 - 1 - 1 = -1.   \nonumber
\end{eqnarray}
From Lemma \ref{lem:l}, by using (\ref{rel:3-23}) and
(\ref{rel:3-24}), the result follows.
\end{proof}

\begin{remark}\label{rmk:6}
Choose orthonormal vectors $X,Y\in T_p \overline{M}$, $p\in
\overline{M}$ such that $Y = JX$ and $X$ is a singular vector. i.e.,
conveniently, $JX = J_1X$ (See \cite{BS1}). From (\ref{rel:3-22}),
we get
$$
\overline{K}(X\wedge Y) = 1 + 3 + 3 + 1 + 0 = 8.
$$
So, the upper bound of the function $\overline{K}(X\wedge Y)$ is
rigid.
\end{remark}

\begin{corollary}
Let $(M, g) = (B\times_f F, g_B + f^2 g_F)$ be a warped product
manifold and $(\overline{M}, \overline{g}) = G_2 (\mathbb{C}^{m+2})
= SU_{m+2}/S(U_m U_2)$ the complex two-plane Grassmannian. Let $\psi
: (M, g)\mapsto (\overline{M}, \overline{g})$ be an isometric
immersion. Assume that the manifold $(M, g)$ is a 2-dimensional
totally geodesic $J$-invariant submanifold of $(\overline{M},
\overline{g})$ with a singular vector field $X\in \Gamma(TM)$ (i.e.,
$J(TM) = TM$).

Then we get
$$
m_1 \cdot 8 \leq \frac{\triangle f}{f} \leq m_1 \cdot 8.
$$
\end{corollary}

\begin{proof}
By Lemma \ref{lem:l} and (\ref{rel:3-22}), we obtain the result.
\end{proof}

\begin{lemma}
Let $(M, g) = (B\times_f F, g_B + f^2 g_F)$ be a warped product
manifold and $(\overline{M}, \overline{g}) = G_2 (\mathbb{C}^{m+2})
= SU_{m+2}/S(U_m U_2)$ the complex two-plane Grassmannian. Let $\psi
: (M, g)\mapsto (\overline{M}, \overline{g})$ be an isometric
immersion.

There does not exist a 2-dimensional totally geodesic $J$-invariant
submanifold $(M, g)$ of $(\overline{M}, \overline{g})$ with a
singular vector field $X\in \Gamma(TM)$ such that either the warping
function $f$ is not an eigen-function or the eigenvalue of $f$ is
not equal to $8m_1$.
\end{lemma}

\begin{theorem}\label{thm:7}
Let $(M, g) = (B\times_f F, g_B + f^2 g_F)$ be a warped product
manifold and $(\overline{M}, \overline{g}) = SU_{2,m}/S(U_2\cdot
U_m)$ the complex hyperbolic two-plane Grassmannian. Let $\psi : (M,
g)\mapsto (\overline{M}, \overline{g})$ be an isometric immersion.
Then we obtain
\begin{equation}\label{rel:3-25}
\frac{m_1 m^2}{2(m-1)} H^2 - \frac{m_1}{2} || h ||^2 - 4m_1  \leq
\frac{\triangle f}{f} \leq \frac{m^2}{4m_2} H^2 + \frac{1}{2}m_1,
\end{equation}
where $m_1 = \dim B$ and $m_2 = \dim F$ with $m=m_1+m_2$.
\end{theorem}

\begin{proof}
We know that the Riemannian curvature tensor $\overline{R}$
\cite{P2} of $(\overline{M}, \overline{g})$ is given by
\begin{eqnarray}
&&\overline{R}(X,Y)Z = -\frac{1}{2}(\overline{g}(Y,Z)X -
\overline{g}(X,Z)Y
\label{rel:3-26}   \\
&&+ \overline{g}(JY, Z)JX - \overline{g}(JX, Z)JY -
2\overline{g}(JX, Y)JZ \nonumber  \\
&&+ \sum_{\alpha=1}^3 (\overline{g}(J_{\alpha} Y, Z)J_{\alpha} X -
\overline{g}(J_{\alpha} X, Z)J_{\alpha} Y - 2\overline{g}(J_{\alpha}
X, Y)J_{\alpha} Z) \nonumber  \\
&&+ \sum_{\alpha=1}^3 (\overline{g}(J_{\alpha} JY, Z)J_{\alpha} JX -
\overline{g}(J_{\alpha} JX, Z)J_{\alpha} JY)) \nonumber
\end{eqnarray}
for $X,Y,Z\in \Gamma(T\overline{M})$.

Hence, in a similar way with Theorem \ref{thm:6}, we easily get the
result.
\end{proof}

\begin{remark}
We choose orthonormal vectors $X,Y\in T_p \overline{M}$, $p\in
\overline{M}$, such that $Y = JX$ and $X$ is a singular vector.
i.e., conveniently, $JX = J_1X$ (See \cite{BS2}). In a similar way
with Remark \ref{rmk:6}, we obtain
$$
\overline{K}(X\wedge Y) = -4.
$$
So, the lower bound of the function $\overline{K}(X\wedge Y)$ is
rigid.
\end{remark}

\begin{corollary}
Let $(M, g) = (B\times_f F, g_B + f^2 g_F)$ be a warped product
manifold and $(\overline{M}, \overline{g}) = SU_{2,m}/S(U_2\cdot
U_m)$ the complex hyperbolic two-plane Grassmannian. Let $\psi : (M,
g)\mapsto (\overline{M}, \overline{g})$ be an isometric immersion.
Assume that the manifold $(M, g)$ is a 2-dimensional totally
geodesic $J$-invariant submanifold of $(\overline{M}, \overline{g})$
with a singular vector field $X\in \Gamma(TM)$

Then we get
$$
m_1 \cdot (-4) \leq \frac{\triangle f}{f} \leq m_1 \cdot (-4).
$$
\end{corollary}

\begin{lemma}
Let $(M, g) = (B\times_f F, g_B + f^2 g_F)$ be a warped product
manifold and $(\overline{M}, \overline{g}) = SU_{2,m}/S(U_2\cdot
U_m)$ the complex hyperbolic two-plane Grassmannian. Let $\psi : (M,
g)\mapsto (\overline{M}, \overline{g})$ be an isometric immersion.

There does not exist a 2-dimensional totally geodesic $J$-invariant
submanifold $(M, g)$ of $(\overline{M}, \overline{g})$ with a
singular vector field $X\in \Gamma(TM)$ such that either the warping
function $f$ is not an eigen-function or the eigenvalue of $f$ is
not equal to $-4m_1$.
\end{lemma}

\begin{theorem}\label{thm:8}
Let $(M, g) = (B\times_f F, g_B + f^2 g_F)$ be a warped product
manifold and $(\overline{M}, \overline{g}) = Q^m = SO_{m+2}/SO_m
SO_2$ the complex quadric. Let $\psi : (M, g)\mapsto (\overline{M},
\overline{g})$ be an isometric immersion. Then we get
\begin{equation}\label{rel:3-27}
\frac{m_1 m^2}{2(m-1)} H^2 - \frac{m_1}{2} || h ||^2 -2.3 m_1  \leq
\frac{\triangle f}{f} \leq \frac{m^2}{4m_2} H^2 + 5m_1,
\end{equation}
where $m_1 = \dim B$ and $m_2 = \dim F$ with $m=m_1+m_2$.
\end{theorem}

\begin{proof}
We see that the Riemannian curvature tensor $\overline{R}$ \cite{S}
of $(\overline{M}, \overline{g})$ is given by
\begin{eqnarray}
&&\overline{R}(X,Y)Z = \overline{g}(Y,Z)X - \overline{g}(X,Z)Y
\label{rel:3-28}   \\
&&+ \overline{g}(JY, Z)JX - \overline{g}(JX, Z)JY -
2\overline{g}(JX, Y)JZ \nonumber  \\
&&+ \overline{g}(AY, Z)AX - \overline{g}(AX, Z)AY
 + \overline{g}(JAY, Z)JAX -
\overline{g}(JAX, Z)JAY \nonumber
\end{eqnarray}
for $X,Y,Z\in \Gamma(T\overline{M})$. Given orthonormal vectors
$X,Y\in T_p \overline{M}$, $p\in \overline{M}$, we obtain
\begin{eqnarray}
&&\overline{K}(X\wedge Y) = \overline{R}(X,Y,Y,X) = 1 +
3\overline{g}(JX, Y)^2 \label{rel:3-29}   \\
&&+ \overline{g}(AY, Y)\overline{g}(AX, X) - \overline{g}(AX, Y)^2 +
\overline{g}(JAY, Y) \overline{g}(JAX, X) - \overline{g}(JAX, Y)^2.
\nonumber
\end{eqnarray}
Since $A$ is an involution (i.e., $A^2 = id$), we get the following
decompositions
$$
\left.
  \begin{array}{ll}
     &X = a\overline{X}_1 + b\overline{X}_2  \\
     &Y = c\overline{Y}_1 + d\overline{Y}_2,
  \end{array}
\right.
$$
where $|\overline{X}_1| = |\overline{X}_2| = |\overline{Y}_1| =
|\overline{Y}_2| = 1$, $\overline{X}_1,\overline{Y}_1\in V(A) = \{
Z\in T_p \overline{M} \mid AZ = Z \}$,
$\overline{X}_2,\overline{Y}_2\in JV(A)$ (See \cite{S}) so that
\begin{align*}
  &1 = |X|^2 = a^2 + b^2,    \\
  &1 = |Y|^2 = c^2 + d^2,    \\
  &0 = \overline{g}(X, Y) = ac\overline{g}(\overline{X}_1,
\overline{Y}_1) + bd\overline{g}(\overline{X}_2, \overline{Y}_2).
\end{align*}
\noindent Conveniently, let $(a,b) = (\cos \alpha, \sin \alpha)$ and
$(c,d) = (\cos \beta, \sin \beta).$

If necessary, by replacing
$\overline{X}_1,\overline{X}_2,\overline{Y}_1,\overline{Y}_2$ with
$-\overline{X}_1,-\overline{X}_2,-\overline{Y}_1,-\overline{Y}_2$,
respectively, we may assume
\begin{equation}\label{rel:3-30}
0\leq \alpha,\beta \leq \frac{\pi}{2}.
\end{equation}
Thus, with a simple calculation, we have
\begin{eqnarray}
&&\overline{K}(X\wedge Y) = 1 + 2\overline{a}^2 \cos^2 \alpha \sin^2 \beta
+ 2\overline{b}^2 \sin^2 \alpha \cos^2 \beta + \cos 2\alpha \cos 2\beta \label{rel:3-31}   \\
&&+ 2\overline{a}\overline{b} \sin 2\alpha \sin 2\beta +
\overline{c}\overline{d} \sin 2\alpha \sin 2\beta - \overline{e}^2
\cos^2 \alpha \cos^2 \beta, \nonumber
\end{eqnarray}
where
\begin{align*}
  &\overline{a} = \overline{g}(\overline{X}_1, J\overline{Y}_2)   \\
  &\overline{b} = \overline{g}(\overline{X}_2, J\overline{Y}_1)   \\
  &\overline{c} = \overline{g}(J\overline{Y}_1, \overline{Y}_2)   \\
  &\overline{d} = \overline{g}(J\overline{X}_1, \overline{X}_2)   \\
  &\overline{e} = \overline{g}(\overline{X}_1, \overline{Y}_1).
\end{align*}

\begin{figure}[t]
\noindent
\begin{minipage}{2.5in}
\centerline {\epsfig{figure = 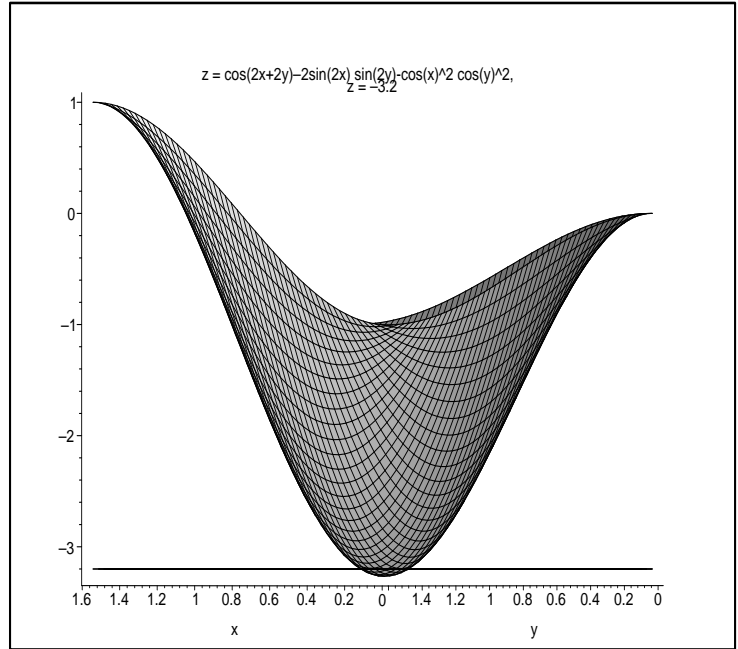, width=3.4in,height=3.8in,
angle=-90 }} \centerline{(a) $z = h(x,y)$ and $z = -3.2$}
\end{minipage} \hfill
\begin{minipage}{2.5in}
\centerline {\epsfig{figure = 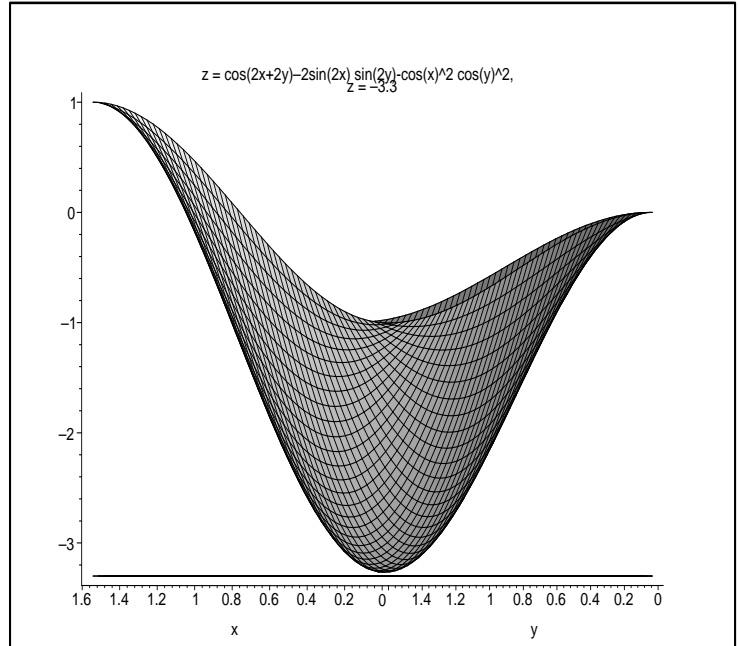, width=3.4in,height=3.8in,
angle=-90 }} \centerline {(b) $z = h(x,y)$ and $z = -3.3$}
\end{minipage}
\caption{ The lower bound of $h(x,y)$} \label{fig}
\end{figure}

We see
\begin{equation}\label{rel:3-32}
-1\leq
\overline{a},\overline{b},\overline{c},\overline{d},\overline{e}
\leq 1.
\end{equation}
Consider the function
\begin{eqnarray}
&&S(x,y) = 2\overline{a}^2 \cos^2 x \sin^2 y +
2\overline{b}^2 \sin^2 x \cos^2 y + \cos 2x \cos 2y  \label{rel:3-33}   \\
&&+ 2\overline{a}\overline{b} \sin 2x \sin 2y +
\overline{c}\overline{d} \sin 2x \sin 2y - \overline{e}^2 \cos^2 x
\cos^2 y \nonumber
\end{eqnarray}
for $(x,y)\in [0, \frac{\pi}{2}]\times [0, \frac{\pi}{2}]$.

Since $\sin 2x \sin 2y \geq 0$, by (\ref{rel:3-32}), we obtain
\begin{eqnarray}
&&S(x,y) \leq  2\cos^2 x \sin^2 y + 2\sin^2 x \cos^2 y   \label{rel:3-34}   \\
&&+ \cos 2x \cos 2y +
2\sin 2x \sin 2y + \sin 2x \sin 2y  \nonumber   \\
&&=  2(\cos x \sin y + \sin x \cos y)^2 + \cos (2x-2y) + \sin 2x \sin 2y  \nonumber   \\
&&=  2\sin^2 (x+y) + \cos (2x-2y) + \sin 2x \sin 2y  \nonumber   \\
&&\leq  4 \nonumber
\end{eqnarray}
and
\begin{eqnarray}
&&S(x,y) \geq  \cos 2x \cos 2y - 2\sin 2x \sin 2y  \label{rel:3-35}   \\
&&- \sin 2x \sin 2y - \cos^2 x \cos^2 y \nonumber   \\
&&=  \cos (2x+2y) - 2\sin 2x \sin 2y - \cos^2 x \cos^2 y. \nonumber
\end{eqnarray}
Consider the function $h(x,y) = \cos (2x+2y) - 2\sin 2x \sin 2y -
\cos^2 x \cos^2 y$ for $(x,y)\in [0, \frac{\pi}{2}]\times [0,
\frac{\pi}{2}]$.

We see
\begin{equation}\label{rel:3-36}
h(x,y) \geq  -3.3 \ \text{(See Figure \ref{fig}).}
\end{equation}
From Lemma \ref{lem:l}, by using (\ref{rel:3-31}), (\ref{rel:3-33}),
(\ref{rel:3-34}), (\ref{rel:3-35}), and (\ref{rel:3-36}), the result
follows.
\end{proof}

\begin{remark}\label{rmk:8}
We get $h(\frac{\pi}{4},\frac{\pi}{4}) = -3.25$. But $h_x
(\frac{\pi}{4},\frac{\pi}{4}) = \frac{1}{2} \neq 0$ and $h_y
(\frac{\pi}{4},\frac{\pi}{4}) = \frac{1}{2} \neq 0$, which implies
that $(\frac{\pi}{4},\frac{\pi}{4})$ is not a critical point of
$h(x,y)$.
\end{remark}

\begin{corollary}
Let $(M, g) = (B\times_f F, g_B + f^2 g_F)$ be a warped product
manifold and $(\overline{M}, \overline{g}) = Q^m = SO_{m+2}/SO_m
SO_2$ the complex quadric. Let $\psi : (M, g)\mapsto (\overline{M},
\overline{g})$ be an isometric immersion. Assume that the manifold
$(M, g)$ is a 2-dimensional totally geodesic $J$-invariant
submanifold of $(\overline{M}, \overline{g})$ with a non-vanishing
vector field $X\in \Gamma(TM)\cap V(A)$.

Then we get
$$
m_1 \cdot 2 \leq \frac{\triangle f}{f} \leq m_1 \cdot 2.
$$
\end{corollary}

\begin{proof}
By Lemma \ref{lem:l} and (\ref{rel:3-29}), we obtain the result.
\end{proof}

\begin{lemma}
Let $(M, g) = (B\times_f F, g_B + f^2 g_F)$ be a warped product
manifold and $(\overline{M}, \overline{g}) = Q^m = SO_{m+2}/SO_m
SO_2$ the complex quadric. Let $\psi : (M, g)\mapsto (\overline{M},
\overline{g})$ be an isometric immersion.

There does not exist a 2-dimensional totally geodesic $J$-invariant
submanifold $(M, g)$ of $(\overline{M}, \overline{g})$ with a
non-vanishing vector field $X\in \Gamma(TM)\cap V(A)$ such that
either the warping function $f$ is not an eigen-function or the
eigenvalue of $f$ is not equal to $2m_1$.
\end{lemma}

\begin{corollary}
Let $(M, g) = (B\times_f F, g_B + f^2 g_F)$ be a warped product
manifold and $(\overline{M}, \overline{g}) = Q^m = SO_{m+2}/SO_m
SO_2$ the complex quadric. Let $\psi : (M, g)\mapsto (\overline{M},
\overline{g})$ be an isometric immersion. Assume that the manifold
$(M, g)$ is a 2-dimensional totally geodesic
 submanifold of $(\overline{M}, \overline{g})$ with $TM \subset V(A)$.

Then we get
$$
m_1 \cdot 2 \leq \frac{\triangle f}{f} \leq m_1 \cdot 2.
$$
\end{corollary}

\begin{lemma}
Let $(M, g) = (B\times_f F, g_B + f^2 g_F)$ be a warped product
manifold and $(\overline{M}, \overline{g}) = Q^m = SO_{m+2}/SO_m
SO_2$ the complex quadric. Let $\psi : (M, g)\mapsto (\overline{M},
\overline{g})$ be an isometric immersion.

There does not exist a 2-dimensional totally geodesic
 submanifold $(M, g)$ of $(\overline{M}, \overline{g})$ with $TM \subset V(A)$ such that
either the warping function $f$ is not an eigen-function or the
eigenvalue of $f$ is not equal to $2m_1$.
\end{lemma}

\begin{corollary}
Let $(M, g) = (B\times_f F, g_B + f^2 g_F)$ be a warped product
manifold and $(\overline{M}, \overline{g}) = Q^m = SO_{m+2}/SO_m
SO_2$ the complex quadric. Let $\psi : (M, g)\mapsto (\overline{M},
\overline{g})$ be an isometric immersion. Assume that the manifold
$(M, g)$ is a 2-dimensional totally geodesic
 submanifold of $(\overline{M}, \overline{g})$ with $TM \perp J(TM)$ and $\dim (TM \cap V(A)) = \dim (TM \cap JV(A)) = 1$.

Then we get
$$
m_1 \cdot 0 \leq \frac{\triangle f}{f} \leq m_1 \cdot 0.
$$
\end{corollary}

\begin{lemma}
Let $(M, g) = (B\times_f F, g_B + f^2 g_F)$ be a warped product
manifold and $(\overline{M}, \overline{g}) = Q^m = SO_{m+2}/SO_m
SO_2$ the complex quadric. Let $\psi : (M, g)\mapsto (\overline{M},
\overline{g})$ be an isometric immersion.

There does not exist a 2-dimensional totally geodesic
 submanifold $(M, g)$ of $(\overline{M}, \overline{g})$ with $TM \perp J(TM)$ and $\dim (TM \cap V(A)) = \dim (TM \cap JV(A)) = 1$ such that
 the warping function $f$ is not a harmonic function.
\end{lemma}

\section{Open questions}\label{esti3}

In section 3 and section 4, we deal with the estimates of the
functions $\frac{\triangle f}{f}$ for isometric immersions $\psi :
(M, g) = (B\times_f F, g_B + f^2 g_F) \mapsto (\overline{M},
\overline{g})$. And we also consider equality cases and their
applications. As future projects, we can use these results to study
the properties of base manifolds and target manifolds and
investigate another equality cases and their applications. We also
estimate the functions $\frac{\triangle f}{f}$ by changing target
manifolds.

\vspace{2mm}

{\bf Questions}

1. What kind of eigenvalues of the warping functions $f$ can we get?

(We obtained the following eigenvalues:

$\displaystyle{m_1 c, \frac{m_1 c}{4}, \frac{m_1 (c+3)}{4}, m_1,
\frac{m_1 (c-3)}{4}, 8m_1, -4m_1, 2m_1, 0}$.)

2. If the warping function $f$ is an eigen-function with eigenvalue
$d$, then what can we say about $M$ and $\overline{M}$?



\end{document}